\documentclass[10pt,a4paper,smallcondensed]{article}
\usepackage[latin1]{inputenc}
\usepackage[T1]{fontenc}
\usepackage{amsmath,amscd}
\usepackage{graphicx}
\usepackage[all]{xy}
\usepackage{amsthm}
\usepackage{amssymb,url}
\usepackage{bbm}
\usepackage{ae}
\usepackage[all]{xy}

\newtheorem{sats}{Theorem}[section]
\newtheorem{sats*}{Theorem}

\newtheorem{lem}[sats]{Lemma}
\newtheorem{cor}[sats]{Corollary}
\newtheorem{prop}[sats]{Proposition}

\newcommand{\R}{\mathbbm{R}}
\newcommand{\C}{\mathbbm{C}}

\newcommand{\Z}{\mathbbm{Z}}
\newcommand{\Q}{\mathbbm{Q}}

\newcommand{\ellL}{\mathcal{L}}

\newcommand{\cc}{ \mathrm{cc}}

\newcommand{\ch}{ \mathrm{ch}}
\newcommand{\id}{ \mathrm{id}}
\newcommand{\rd}{\mathrm{d}}

\newcommand{\Ko}{\mathcal{K}}

\newcommand{\Bo}{\mathcal{B}}

\newcommand{\He}{\mathcal{H}}

\newcommand{\Eg}{\mathcal{E}}

\newcommand{\ind}{\mathrm{i} \mathrm{n} \mathrm{d} \,}

\newcommand{\Hom}{\mathrm{Hom}}
\newcommand{\End}{\mathrm{End}}
\renewcommand{\epsilon}{\varepsilon}
\renewcommand{\phi}{\varphi}

\newcommand{\sign}{\mathrm{sign} }

\newcommand{\e}{\mathrm{e}}
\newcommand{\suptra}{\mathrm{st}\mathrm{r}}
\newcommand{\tra}{\mathrm{t}\mathrm{r}}

\title{Analytic formulas for topological degree of 
non-smooth mappings: the even-dimensional case}
\author{Magnus Goffeng}
\date{Department of Mathematical Sciences, Division of Mathematics\\
Chalmers university of Technology and University of Gothenburg}

\begin{document}
\maketitle

\bibliographystyle{amsplain}

\begin{abstract}
Topological degrees of continuous mappings between oriented manifolds of even dimension are studied in terms of index theory of pseudo-differential operators. The index formalism of non-commutative geometry is used to derive analytic integral formulas for the index of a $0$:th order pseudo-differential operator twisted by a H\"older continuous complex vector bundle. The index formula gives an analytic formula for the degree of a H\"older continuous mapping between even-dimensional oriented manifolds. The paper is an independent continuation of the paper \emph{Analytic formulas for topological degree of non-smooth mappings: the odd-dimensional case}.
\end{abstract}
\normalsize

\section*{Introduction}

This paper is an independent continuation of the paper \cite{mgoffodd} where the degree of a mapping from the boundary of a strictly pseudo-convex domain was given in terms of an explicit integral formula involving the Szeg\"o kernel as well as the more computable Henkin-Ramirez kernel. In this paper analytic formulas are given for general even-dimensional oriented manifolds in terms of the signature operator. The classical approach to mapping degree is to define in an abstract way the degree of a continuous mapping between two compact connected oriented manifolds of the same dimension in terms of cohomology. If the function $f$ is differentiable, an analytic formula for the degree can be derived using Brouwer degree, see \cite{reyeche}, or the more global picture of de Rham cohomology. Without differentiability conditions on $f$, the only known analytic degree formula beyond \cite{mgoffodd} is a formula of Connes which only holds in one dimension, see more in Chapter III.$2.\alpha$ of \cite{connes}. Our aim is to find another formula for the degree, that is valid for a H\"older continuous function, by expressing the degree as the index of a pseudo-differential operator and using the approach of \cite{connes} and \cite{mgoffodd}. 

Throughout the paper we will use the idea that the Chern character extracts cohomological information of a continuous mapping $f:X\to Y$ between even dimensional manifolds from the induced mapping $f^*:K^0(Y)\to K^0(X)$. The $K$-theory is a topological invariant and the picture of the index mapping as a pairing in a local cohomology theory via Chern characters in the Atiyah-Singer index theorem can be applied to more general classes of functions than the smooth functions. The cohomology theory present throughout all the index theory is the cyclic homology. For a H\"older continuous mapping $f:X\to Y$ of exponent $\alpha$ and an elliptic differential operator $A$ this idea can be read out from the commutativity of the diagram:
\begin{equation}
\label{holdiaeven}
\begin{CD}
K_0(C^\infty(Y))@>f^*>>K_0(C^\alpha(X))@>\ind_A >> \Z \\
@VV\ch_YV @VV\ch_{X}V @VVV\\
HC_{even}(C^\infty(Y))@>f^*>>HC_{even}(C^\alpha(X)) @>\tilde{\mu}_{A} >> \C \\
\end{CD} 
\end{equation}
where the mapping $\tilde{\mu}_{A}:HC_{even}(C^\alpha(X))\to \C$ is a cyclic cocycle on $C^\alpha(X)$ defined as the Connes-Chern character of the bounded $K$-homology class that $A$ defines, see more in \cite{connesncdg} and \cite{connes}. The right-hand side of the diagram \eqref{holdiaeven} is commutative by Connes' index formula, see Proposition $4$ of Chapter IV$.1$ of \cite{connes}. The dimension in which the Chern character will take values depends on the H\"older exponent $\alpha$. More explicitly, for $2n$-dimensional manifolds, the cocycle $\tilde{\mu}_{A}$ can be chosen as a cyclic $2k$-cocycle for any $k>n/\alpha$.

To describe this idea more explicitely, when $E\to X$ is a smooth complex vector bundle defined by the smooth projection-valued function $p:X\to \Ko$, where $\Ko$ denotes the $C^*$-algebra of compact operators on a complex separable Hilbert space, the index of the twisted differential operator $A_E:=p(A\otimes \id_\Ko)p$ can be calculated in terms of the de Rham cohomology using the Atiyah-Singer index formula as 
\[\ind A_E= \int_{T^*X}\pi^*\ch[E]\wedge \ch[A]\wedge \pi^*Td(X),\]
where $Td(X)$ denotes the Todd class of the complexified tangent bundle. In particular,  if $E\to Y$ is a vector bundle with fibers of complex dimension $N$ on an even-dimensional manifold $Y$ such that $\ch[E]$ only contains a constant term and a top-degree term and $f:X\to Y$ is smooth we can consider the vector bundle $f^*E\to X$. Naturality of the Chern character implies the identity 
\[\deg (f)\,\ch_0[A]\int_Y\ch_{Y}[E]=\ind A_{f^*E}-N\cdot\ind(A).\] 
In Theorem \ref{degsatseven}, we construct an explicit $2^{n-1}$-dimensional vector bundle $E_Y$ over a $2n$-dimensional compact oriented manifold $Y$ satisfying the above conditions together with the condition $\int_Y\ch_{Y}[E_Y]=1$. In the correct analytic setting the above degree formula extends to H\"older continuous functions. The analytic setting we choose in Theorem \ref{psumfred} is to associate a Fredholm module $(\tilde{\pi},\tilde{F}_A)$ with an elliptic differential operator $A$ of positive order. The Fredholm module $(\tilde{\pi},\tilde{F}_A)$ is $q$-summable over the algebra of H\"older continuous functions $C^\alpha(X)$, for any $q>\dim(X)/\alpha$. Thus the Connes-Chern character $\tilde{\mu}_A:=\cc_k(\tilde{\pi},\tilde{F}_A)$ is well defined for dimensions $2k>\dim(X)/\alpha$. In Theorem \ref{holderdegreeeven} we take $A$ to be the signature operator and show that if $f:X\to Y$ is a H\"older continuous mapping between two oriented $2n$-dimensional manifolds the following analytic degree formula holds: 
\[\deg (f)=2^{-n}\tilde{\mu}_{A}(\ch_{X}f^*([E_Y]-2^{n-1}[1])).\]

The drawback exhibited in \cite{mgoffodd}, where the results were restricted to boundaries of strictly pseudo-convex domains in Stein manifolds, is not present in this paper. The restriction that $X$ and $Y$ must be even-dimensional does not really pose a problem since when $X$ and $Y$ are odd-dimensional we can consider the mapping $f\times \id:X\times S^1\to Y\times S^1$ instead which is a mapping between even-dimensional manifolds and $\deg(f)=\deg(f\times \id)$. The drawback of the degree formula in Theorem \ref{holderdegreeeven} is that it is in general quite hard to calculate explicit integral kernels for pseudo-differential operators.  \\

\section{$K$-theory and Connes' index formula}

To formulate the calculation of mapping degrees in a setting fitting with non-smooth mappings, we need a framework for "differential geometry" where there are no classical differentials. The framework we will use is Alain Connes' non-commutative geometry, see \cite{connesncdg} and \cite{connes}. We will recall some basic concepts of non-commutative geometry in this section.

The $K$-theory of a compact topological space $Y$ is defined as the Grothendieck group of the abelian semigroup of isomorphism classes of complex vector bundles under direct sum. We will tacitly assume that all vector bundles are complex throughout the paper. The Serre-Swan theorem states a one-to-one correspondence between the isomorphism classes of vector bundles over a compact space $Y$ and projection valued functions $p:Y\to \Ko$, see \cite{swan}. Here $\Ko$ denotes the $C^*$-algebra of compact operators on some separable, infinite dimensional Hilbert space. Following the Serre-Swan theorem, an equivalent approach to $K$-theory is to use equivalence classes of projections $p\in C(Y)\otimes \Ko$. The $K$-theory is denoted by $K_0(C(Y))$. To read more about $K$-theory, see \cite{blacker}. The formulation of $K$-theory in terms of projections can be defined for any algebra $\mathcal{A}$ as formal differences of equivalence classes of projections $p\in \mathcal{A}\otimes \Ko$. 

Clearly, the abelian group $K_0(\mathcal{A})$ depends covariantly on the algebra $\mathcal{A}$ so $K_0$ defines a functor. In particular, the functor $K_0$ has many properties making the $K$-theory of a $C^*$-algebra manageable to calculate, for instance; homotopy invariance, half exactness and stability under tensoring by a matrix algebra. Furthermore, a dense embedding of topological algebras $\mathcal{A}'\hookrightarrow \mathcal{A}$ which is isoradial induces an isomorphism on $K$-theory, see more in \cite{cumero}. For instance, if $Y$ is a compact manifold all of the embeddings $C^\infty(Y)\subseteq C^\alpha(Y)\subseteq C(Y)$ induce isomorphisms on $K$-theory. Here $C^\alpha(Y)$ denotes the algebra of H\"older continuous functions of exponent $\alpha\in]0,1]$. The isomorphism $K_0(C(Y))\cong K_0(C^\infty(Y))$ enables us to define the Chern character $\ch:K_0(C(Y))\to H^{even}_{dR}(Y)$ on classes $[p]\in K_0(C(Y))$ represented by a smooth $p:Y\to \Ko$ as
\[\ch[p]:=\sum_{j=0}^\infty \frac{1}{(2\pi i)^j j!}\tra (p\rd p\rd p)^j.\]
We choose the trace as the fiberwise operator trace in $\wedge^*T^*Y\otimes \Ko$ which is well defined since a compact projection is of finite rank. The term in the sum of degree $2j$ is denoted by $\ch_j[p]$. This definition of the Chern character is motivated by Chern-Weil theory since the curvature of the Levi-Civita connection on the vector bundle associated with a smooth self-adjoint projection $p$ is precisely $p\rd p\rd p$. The Chern character extends to an additive mapping on all formal differences in $K_0(C(Y))$ due to the universal property of the Grothendieck group.

However, we will need a Chern character defined on H\"older continuous projections. The cohomology theory fitting with index theory of more complicated geometries than smooth functions on smooth manifolds is cyclic homology. We will consider Connes' original definition of cyclic homology which simplifies the construction of the Chern character and the Chern-Connes character. We will let $\mathcal{A}$ denote a topological algebra and we will use the notation $\mathcal{A}^{\otimes k}$ for the $k$-th tensor power of $\mathcal{A}$. The Hochschild differential $b:\mathcal{A}^{\otimes k}\to \mathcal{A}^{\otimes k-1}$ is defined by 
\begin{align*}
b(x_0\otimes x_1\otimes \cdots\otimes x_k\otimes x_{k+1})&:=(-1)^{k+1}x_{k+1}x_0\otimes x_1\otimes \cdots \otimes x_k+\\
+\sum_{j=0}^{k}(-1)^j&x_0\otimes  \cdots \otimes  x_{j-1}\otimes  x_jx_{j+1}\otimes x_{j+2}\otimes \cdots \otimes x_{k+1}.
\end{align*}
The cyclic permutation operator $\lambda:\mathcal{A}^{\otimes k}\to \mathcal{A}^{\otimes k}$ is defined as
\[\lambda(x_0\otimes x_1\otimes \cdots \otimes x_k)=(-1)^kx_k\otimes x_0\otimes \cdots \otimes x_{k-1}.\]
We define a complex of $\C$-vector spaces $C^\lambda_*(\mathcal{A})$ by 
\[C^\lambda_k(\mathcal{A}):=\mathcal{A}^{\otimes k+1}/(1-\lambda)\mathcal{A}^{\otimes k+1},\]
with differential given by $b$. The homology of the complex $C^\lambda_*(\mathcal{A})$ is called the cyclic homology of $\mathcal{A}$ and will be denoted by $HC_*(\mathcal{A})$. A cycle in $C_k^\lambda(\mathcal{A})$ will be called a cyclic $k$-cycle.

The complex $C^k_\lambda(\mathcal{A})$ is defined as the space of continuous linear functionals $\mu$ on $\mathcal{A}^{\otimes k+1}$ such that $\mu\circ\lambda=\mu$. The Hochschild coboundary operator $\mu\mapsto \mu\circ b$ makes $C^*_\lambda(\mathcal{A})$ into a complex. The cohomology of the complex $C^*_\lambda(\mathcal{A})$ will be denoted by $HC^*(\mathcal{A})$ and is called the cyclic cohomology of $\mathcal{A}$. Cyclic cohomology is an algebraic generalization of de Rham homology. The difference lies in that the dimension defines a grading on the de Rham theories, while the dimension defines a filtration on the cyclic theories. This difference can be explained by a Theorem of Connes \cite{connes} stating that if $X$ is a compact oriented manifold, there is an isomorphism 
\begin{equation}
\label{cycde}
HC^k(C^\infty(X))\cong Z_k(X)\oplus \bigoplus_{j> 0} H^{dR}_{k-2j}(X),
\end{equation}
where $Z_k(X)$ denotes the space of closed $k$-currents on $X$. The filtration on cyclic cohomology can be described by the linear mapping $S:HC^k(\mathcal{A})\to HC^{k+2}(\mathcal{A})$ called the periodicity operator. For a definition of the periodicity operator, see \cite{connes}. 

To define the Chern character $\ch_{2k}:K_0(\mathcal{A})\to HC_{2k}(\mathcal{A})$ in degree $2k$ it is sufficient to define it on projections and extending to all $K$-theory classes using the universal property of the Grothendieck group. Following Proposition $3$ of Chapter III$.3$ of \cite{connes} we define the Chern character of a projection $p$ by 
\begin{equation}
\label{cycchern}
\ch_{2k}[p]:=(k!)^{-1}\tra\left(\underbrace{p\otimes p \otimes \cdots \otimes p\otimes p}_{2k+1 \quad\mbox{factors}}\right).
\end{equation}
The additive pairing between $HC^{2k}(\mathcal{A})$ and a $K$-theory element $x\in K_0(\mathcal{A})$ is defined by 
\[\langle \mu, x\rangle_{2k}:=\mu.\ch_{2k}[x].\]
The choice of normalization implies that for a cohomology class in $HC^{2k}(\mathcal{A})$ represented by the cyclic cocycle $\mu$, the pairing satisfies
\[\langle S\mu, x\rangle_{2k+2}=\langle \mu, x\rangle_{2k},\]
see Proposition $3$ in Chapter III$.3$ of \cite{connes}.

The homology theory dual to $K$-theory is $K$-homology. Analytic $K$-homology is described by Fredholm modules, a theory fitting well with index theory and cyclic cohomology. For $q\geq 1$, let $\ellL^q(\He)\subseteq \Bo(\He)$ denote the ideal of Schatten class operators on a separable Hilbert space $\He$, so $T\in \ellL^q(\He)$ if and only if $\tra((T^*T)^{q/2})<\infty$. For $q\neq 2$ there is no exact description for an integral operator to belong to the Schatten class of order $q$. However, for $q>2$ there exists a convenient sufficient condition on the kernel, found in \cite{russo}. We will return to this subject a little later. 

A graded Hilbert space is a Hilbert space $\He$ equipped with an involutive mapping $\gamma$, that is, $\gamma^2=1$. While $\gamma$ is an involution, we can decompose $\He=\He_+\oplus \He_-$, where $\He_\pm=\ker(\gamma\mp1)$. An operator $T$ on $\He$ is called even if $T\He_\pm\subseteq \He_\pm$ and odd if $T\He_\pm\subseteq \He_\mp$. Suppose that $\He$ is a graded Hilbert space and $\pi:\mathcal{A}\to \Bo(\He)$ is an even representation of a trivially graded $\C$-algebra $\mathcal{A}$. If  $F\in \Bo(\He)$ is an odd operator such that 
\begin{equation}
\label{fredprop}
F^2=1,\; F=F^*\quad \mbox{and}\quad [F,\pi(a)]\in \ellL^{q}(\He) \;\forall a\in \mathcal{A},
\end{equation}
the pair $(\pi,F)$ is called a $q$-summable even Fredholm module. The conditions $F^2=1$ and $F=F^*$ simplifies many calculations, but in practice it is sufficient if they hold modulo $q$-summable operators. If we decompose the graded Hilbert space $\He=\He_+\oplus \He_-$ into its even and odd part, the odd operator $F$ decomposes as:
\[F=
\begin{pmatrix}
0& F_+\\
F_-& 0 \end{pmatrix}\]
where $F_+:\He_-\to \He_+$ and $F_-:\He_+\to \He_-$. Similarly we can decompose $\pi=\pi_+\oplus \pi_-$ where $\pi_\pm:\mathcal{A}\to \Bo(\He_\pm)$ are representations. The first and second condition in \eqref{fredprop} are equivalent to the conditions $F_+=F_-^{-1}=F_-^*$ and the commutator condition is equivalent to 
\[F_-\pi_+(a)-\pi_-(a)F_-\in \ellL^{q}(\He_+,\He_-)\quad\mbox{and}\quad F_+\pi_-(a)-\pi_+(a)F_+\in \ellL^{q}(\He_-,\He_+).\] 
If the pair $(\pi,F)$ satisfies the requirement in equation \eqref{fredprop} but with $\ellL^q(\He)$ replaced by $\Ko(\He)$ the pair $(\pi,F)$ is a bounded even Fredholm module. The set of homotopy classes of bounded even Fredholm modules forms an abelian group under direct sum called the even analytic $K$-homology of $\mathcal{A}$ and is denoted by $K^0(\mathcal{A})$. For a more thorough presentation of Fredholm modules, e.g. Chapter VII and VIII of \cite{blacker}.

Following Definition $3$ of Chapter IV$.1$ of \cite{connes} we define the Connes-Chern character $\cc_{2k}(\pi,F)$ of a $q$-summable even Fredholm module $(\pi,F)$ for $2k\geq q$ as the cyclic $2k$-cocycle: 
\[\cc_{2k}(\pi,F)(a_0,a_1,\ldots, a_{2k}):=
(-1)^{k}k!\;\suptra (\pi(a_0)[F,\pi(a_1)]\cdots [F,\pi(a_{2k})]),\]
where $\suptra(T):=\tra(\gamma T)$ for $T\in \ellL^1(\He)$ and $\gamma$ denotes the grading on $\He$. This choice of normalization leads to $S\cc_{2k}(\pi,F)=cc_{2k+2}(\pi,F)$, see Proposition $2$ of Chapter IV$.1$ of \cite{connes}.

If $(\pi,F)$ is a bounded Fredholm module over $\mathcal{A}$ we can define the index mapping $\ind\!_F:K_0(\mathcal{A})\to \Z$, again by defining an additive mapping on the semigroup of projections, as 
\[\ind\!_F[p]:=
\ind((\pi_+\otimes \id)(p)F_+(\pi_-\otimes \id)(p)),\]
where we represent $[p]$ by a finite-dimensional projection $p\in \mathcal{A}\otimes \Ko(\C^N)$ and we consider $(\pi_+\otimes \id)(p)F_+(\pi_-\otimes \id)(p)$ as an operator 
\[(\pi_+\otimes \id)(p)F_+(\pi_-\otimes \id)(p):(\pi_-\otimes \id)(p)\left(\He_-\otimes \C^N\right)\to (\pi_+\otimes \id)(p)\left(\He_+\otimes\C^N\right).\]
The association $[p]\times (\pi,F)\mapsto \ind\!_F(p)$ is homotopy invariant and defines a bilinear pairing $K_0(\mathcal{A})\times K^0(\mathcal{A})\to \Z$ which is non-degenerate after tensoring with $\Q$, see more in \cite{blacker}. To simplify the notation, we suppress the dimension $N$ and identify $(\pi,F)$ with the Fredholm module $(\pi\otimes  \id_{\Ko(\C^N)},F\otimes\id_{\C^N})$.

\begin{sats}[Proposition $4$ of Chapter IV$.1$ of \cite{connes}]
\label{connesindexeven}
If $(\pi,F)$ is a $q$-summable even Fredholm module satisfying the conditions \eqref{fredprop} and $2k\geq q$ the index mapping $\ind\!_F$ can be calculated as
\[\ind\!_F[p]=\langle\cc_k(\pi,F), p\rangle_k.\]
\end{sats}

In Theorem \ref{connesindexeven}, the conditions \eqref{fredprop} on the Fredholm module $(\pi,F)$ requires some caution. If for instance, we remove the condition $F^2\neq 1$ one can choose $F$ such that $\ind (F_+)\neq 0$. On the other hand, if $\pi$ is unital then $\ind(F_+)=\ind\!_F(1)$ but $\langle\cc_k(\pi,F), 1\rangle_k=0$, therefore $\ind (F_+)\neq\langle\cc_k(\pi,F), 1\rangle_k$.

In the context of index theory, the periodicity operator $S$ plays the role of extending index formulas such as that in Theorem \ref{connesindexeven} to larger algebras. Suppose that $\mu$ is a cyclic $k$-cocycle on an algebra $\mathcal{A}_0$ which is a dense $*$-subalgebra of a $C^*$-algebra $\mathcal{A}$. The cyclic $k+2m$-cocycle $S^m\mu$ can sometimes be extended to a cyclic cocycle on a larger $*$-subalgebra $\mathcal{A}_0\subseteq \mathcal{A}'\subseteq \mathcal{A}$. In \cite{mgoffodd} the properties of $S^m\mu$ were studied for $\Omega$ being a strictly pseudo-convex domain in a Stein manifold of complex dimension $n$ and $\mu$ being the cyclic $2n-1$-cocycle on $\mathcal{A}_0=C^\infty(\partial\Omega)$ defined by 
\[\mu:=\sum_{k=0}^n S^k \omega_k,\]
where $\omega_k$ denotes the cyclic $2n-2k-1$-cocycle given by the Todd class $Td_k(\Omega)$ in degree $2k$ as
\[\omega_k(a_0,a_1,\ldots, a_{2n-2k-1}):=\int _{\partial\Omega}a_0 \rd a_1\wedge \rd a_2\wedge \cdots \wedge \rd a_{2n-2k-1}\wedge Td_k(\Omega).\]
It was proved in \cite{mgoffodd} that the cyclic cocycle $S^{m}\mu$ defines the same cyclic cohomology class on $C^\infty(\partial\Omega)$ as 
\begin{align}
\label{oddcycle}
\tilde{\mu}&(a_0,a_1,\ldots, a_{2n+2m-1}):=\\
&\nonumber:=\int_{\partial\Omega^{2n+2m-1}} \tra\left(a_0(z_0)\prod_{j=1}^{2n+2m-1}(a_j(z_{j})-a_j(z_{j-1}))K_{\partial\Omega}(z_{j-1},z_j)\right)\rd V,
\end{align}
where $K_{\partial\Omega}$ denotes the Szeg\"o kernel or the Henkin-Ramirez kernel. The cyclic cocycle $\tilde{\mu}$ is the \emph{odd} Connes-Chern character of the Toeplitz operators on the Hardy space and $\tilde{\mu}$ extends to a cyclic cocycle on the algebra of H\"older continuous functions on $\partial\Omega$. We will later on use a cyclic cocycle of the form $\mu=\cc_{2k}(\pi,F)$ and the periodicity operator to extend a formulation of the Atiyah-Singer index theorem to pseudo-differential operators twisted by a H\"older continuous vector bundle. 

The index formula of Theorem \ref{connesindexeven} holds for $q$-summable Fredholm modules and to deal with the $q$-summability of pseudo-differential operators we will need the following theorem of Russo \cite{russo} to give a sufficient condition for an integral operator to be Schatten class. Following \cite{bepa}, when $X$ is a $\sigma$-finite measure space and $1\leq p,q <\infty$, the mixed $(p,q)$-norm of a function $k:X\times X\to \C$ is defined by 
\[\|k\|_{p,q}:=\left( \int _X\left( \int_X|k(x,y)|^p\rd x \right) ^{\frac{q}{p}}\rd y\right) ^{\frac{1}{q}}.\]
The space of measurable functions with finite mixed $(p,q)$-norm is denoted by $L^{(p,q)}(X\times X)$. By Theorem $4.1$ of \cite{bepa} the space $L^{(p,q)}(X\times X)$ becomes a Banach space in the mixed $(p,q)$-norm which is reflexive if $1<p,q<\infty$. If a bounded operator $K$ has integral kernel $k$, the hermitian conjugate $K^*$ has integral kernel $k^*(x,y):=\overline{k(y,x)}$. 

\begin{sats}[Theorem $1$ in \cite{russo}]
\label{russothmeven}
Suppose that $K:L^2(X)\to L^2(X)$ is a bounded operator given by an integral kernel $k$. If $2<q<\infty$ 
\begin{equation}
\label{russoesteven}
\|K\|_{\ellL^q(L^2(X))}\leq (\|k\|_{q',q}\|k^*\|_{q',q})^{1/2},
\end{equation}
where $q'=q/(q-1)$.
\end{sats}

In the statement of the Theorem in \cite{russo}, the completely unnecessary assumption $k\in L^2(X\times X)$ is made. For the discussion on how to remove the condition $k\in L^2(X\times X)$ and the proof of the next Theorem, we refer to \cite{mgoffodd}.

\begin{sats}
\label{russotracethmeven}
Suppose that $K_j:L^2(X)\to L^2(X)$ are operators with integral kernels $k_j$ for $j=1,\ldots, m$ such that $\|k_j\|_{q',q},\|k_j^*\|_{q',q}<\infty$ for certain $q>2$. Whenever $m\geq q$ the operator $K_1K_2\cdots K_m$ is a trace class operator and we have the trace formula
\[\tra(K_1K_2\cdots K_m)=\int _{X^m}\left(\prod_{j=1}^{m}k_{j}(x_j,x_{j+1})\right)\rd x_1\rd x_2\cdots \rd x_m,\]
where we identify $x_{m+1}$ with $x_1$.
\end{sats}

\section{The virtual projection with Chern character being the volume form}

In order to obtain a formulation of the degree as an index, we start by constructing a vector bundle $E_Y$ over an arbitrary even dimensional oriented manifold $Y$ such that the only non-constant term in $\ch[E_Y]$ is of top degree. The idea is to use the Bott element in $K^0(\R^{2n})$, which defines a virtual rank zero bundle on a coordinate neighborhood in $Y$ and extend this to a virtual bundle on $Y$. 

Let us first recall the construction of the Bott element $\beta\in K^0(\R^{2n})$. The element $\beta$ is represented by the difference class $(\wedge^{ev}_\C\C^n,\wedge ^{odd}_\C \C^n, c)$ where $\wedge^{ev}_\C\C^n$ and $\wedge ^{odd}_\C \C^n$ are considered as trivial vector bundles on $\R^{2n}$ and $c:\R^{2n}\to \Hom (\wedge^{ev}_\C\C^n,\wedge ^{odd}_\C \C^n)$ is constructed by letting $c(x)\in \Hom(\wedge^{ev}_\C\C^n,\wedge ^{odd}_\C \C^n)$ be the operator defined from the complex spin representation and Clifford multiplication by the vector $x\in \R^{2n}$. Since $c(x)$ is invertible for $x\neq 0$, with inverse $c(x)^*/|x|^2$, this difference class is well defined. See more in Chapter $2.7$ of \cite{atiyah} or Part III of \cite{atiyahbottshapiro}. By Proposition $2.7.2$ of \cite{atiyah}, the element $\beta$ generates $K^0(\R^{2n})$. Since $K^0(\R^{2n})=\ker(K^0(S^{2n})\to K^0(\{\infty\})$, the inclusion $\R^{2n}\subseteq S^{2n}$ induces an injection $K^0(\R^{2n})\to K^0(S^{2n})$, and $K^0(S^{2n})$ is generated by the Bott class and the trivial line bundle. Furthermore, the Bott class does as an element of $K^0(S^{2n})$ satisfy that 
\[\ch_{S^{2n}}\beta=[\rd V_{S^{2n}}],\]  
where $[\rd V_{S^{2n}}]\in H^{2n}_{dR}(S^{2n})$ denotes the cohomology class associated with a choice of a normalized volume form. 
 
The problem with this construction of the Bott element is that it does not fit directly into our definition of the Chern character in cyclic cohomology. We will now construct a projection-valued function $p_0:\R^{2n}\to End(\wedge ^*_\C \C^n)=M_{2^n}(\C)$ of rank $2^{n-1}$ that extends to a projection-valued function $p_T$ on $S^{2n}$ such that $\beta=[p_T]-2^{n-1}[1]$ in $K^0(S^{2n})$. Let us identify the complex Clifford algebra $\C l(\R^{2n})$ with $\End(\wedge_\C^*\C^n)$ using the complex spin representation. Define $p_0$ as:
\[p_0(x):=\frac{1}{1+|x|^2}
\begin{pmatrix} 
|x|^2&c(x)\\
c(x)^*& 1\end{pmatrix}\in End(\wedge ^{odd}_\C \C^n\oplus \wedge^{ev}_\C\C^n).\]
While
\[p_0(x)-\begin{pmatrix} 1&0\\0&0\end{pmatrix}=\frac{1}{1+|x|^2}
\begin{pmatrix} 
-1& c(x)\\
c(x)^*& 1\end{pmatrix}=\mathcal{O}(|x|^{-1})\quad \mbox{as}\quad |x|\to \infty,\]
the function $p_0$ extends over infinity to a function $p_T\in C^1(S^{2n},M_{2^n}(\C))$. Let $E_0\to \R^{2n}$ denote the vector bundle associated with $p_0$ using Serre-Swans theorem. One has that 
\[E_0=\{(x,v_1,v_2)\in \R^{2n}\times (\wedge ^{odd}_\C \C^n\oplus \wedge^{ev}_\C\C^n): v_1=c(x)v_2\}.\]
The vector bundle $E_0$ is trivializable via the isomorphism 
\begin{equation}
\label{fiso}
\id\oplus c: \R^{2n}\times \wedge ^{ev}_\C \C^n\to E_0, \quad (x,v)\mapsto (x,c(x)v,v).
\end{equation}
We define the morphism of vector bundles 
\begin{equation}
\label{fhom}
c_0:E_0\to \R^{2n}\times \wedge ^{odd}_\C \C^n, \quad (x,v_1,v_2)\mapsto (x,v_1).
\end{equation}
The morphism $c_0$ is an isomorphism outside the origin, with inverse 
\[(x,v_1)\mapsto (x,v_1,|x|^{-2} c(x)^*v_1).\]

\begin{prop}
\label{chernflatpt}
Under the isomorphism $K^0(S^{2n})\cong K_0(C(S^{2n}))$ the Bott element $\beta$ is mapped to $[p_T]-2^{n-1}[1]$, and therefore $\int_{S^{2n}}\ch_{S^{2n}}[p_T]=1$.
\end{prop}

\begin{proof}
The formal difference class $[p_T]-2^{n-1}[1]\in K_0(C^1(S^{2n}))$ is of virtual rank $0$, so it is in the image of the injection $K^0(\R^{2n})\to K_0(C(S^{2n}))$. The element $[p_T]-2^{n-1}[1]$ clearly comes from the formal difference $[E_0]-2^{n-1}[1]$ which in turn is defined as the difference class $(E_0,\wedge^{odd}_\C \C^n, c_0)\in K^0(\R^{2n})$, where $c_0$ is the bundle morphism of equation \eqref{fhom}. The latter is isomorphic to the Bott class via the isomorphism $\id\oplus c$ defined in equation \eqref{fiso}. It follows that $\ch_{S^{2n}}[p_T]=2^{n-1}+\ch_{S^{2n}}\beta=2^{n-1}+[\rd V_{S^{2n}}]$.
\end{proof}

In the general case, let $Y$ be a compact connected oriented manifold of dimension $2n$. If we take an open subset $U$ of $Y$ with coordinates $(x_{i})_{i=1}^{2n}$ such that 
\[U=\{x:\sum_{i=1}^{2n}|x_{i}(x)|^2<1\}.\]
The coordinates define a diffeomorphism $\nu:U\cong B_{2n}$. Let us also choose a diffeomorphism $\tau:B_{2n}\cong \R^{2n}$.  We can define the projection-valued function $p_Y:Y\to M_{2^{n}}(\C)$ by 
\begin{equation}
\label{localsymboleven}
p_Y(x):=\begin{cases}
p_T(\tau\nu(x))\quad\mbox{for} \quad x\in U\\
p_T(\infty)\quad\mbox{for} \quad x\notin U
\end{cases}
\end{equation}
If we let $\tilde{\nu}:Y\to S^{2n}$ be the Lipschitz continuous function defined by 
\begin{equation}
\label{nuexteven}
\tilde{\nu}(x)=
\begin{cases}
\tau(\nu(x)) \quad\mbox{for} \quad x\in U\\
\infty\quad\mbox{for} \quad x\notin U
\end{cases}
\end{equation}
the Lipschitz function $p_Y$ can be expressed as $p_Y=\tilde{\nu}^*p_T$. Observe that the projection $p_Y$ defines the same class in $K$-theory as the projection $\tilde{p}_Y:=\tilde{\nu}^*p_{T,1}$, where $p_{T,t}(x):=p_T(\e^{t|x|^2}x)$ for $t\geq 0$. For $t>0$ the projection $p_{T,t}$ stabilizes to infinite order at infinity, and therefore $\tilde{p}_Y$ is smooth. We will choose a normalized volume form on $Y$ and let $[\rd V_Y]\in H^{2n}_{dR}(Y)$ denote the associated de Rham cohomology class.

\begin{sats}
\label{degsatseven}
If $Y$ is a compact connected oriented manifold of even dimension and $\rd V_Y$ denotes the normalized volume form on $Y$, the projection $p_Y$ satisfies  
\[\ch[p_Y]=2^{n-1}+[\rd V_Y],\]
in $H^{even}_{dR}(Y)$. Thus, if $f:X\to Y$ is a smooth mapping 
\[\deg(f)=\int_X f^*\ch[p_Y]\]
\end{sats}

\begin{proof}
By Proposition \ref{chernflatpt} we have the identities
\[\int _Y\ch[\tilde{p}_Y]= \int _{U}\ch_{n}[\tilde{p}_Y]= \int _{U}\tilde{\nu}^*\ch_{n}[p_T]=\int _{S^{2n}}\ch_{n}[p_T]=1.\]
Therefore we have the identity $\ch_{n}[\tilde{p}_Y]=[\rd V_Y]$. Since $\ch[p_Y]=\ch[\tilde{p}_Y]$ and $\ch[\tilde{p}_Y]=2^{n-1}+\ch_{n}[\tilde{p}_Y]$ up to an exact form on $U$ and vanishes to first order at $\partial U$ the Theorem follows.
\end{proof}

Later on we will also need the Chern character of $p_Y$ in cyclic homology as is defined in \eqref{cycchern}. First we need a lemma. We will use the notation $\langle\cdot,\cdot\rangle$ for the scalar product in $\R^{2n}$. For an orthogonal basis $e_1,e_2,\ldots , e_{2n}$ of $\R^{2n}$ the Clifford algebra $\C l(\R^{2n})$ has a basis consisting of multiples $e_{j_1}\cdots e_{j_l}$ for $1\leq j_1<\ldots <j_l\leq 2n$. Any element $u$ in the complex tensor algebra of $\R^{2n}$ does by the universal property of the Clifford algebras define an element $\tilde{u}\in \C l(\R^{2n})$. For a tensor $u$ we let $[u]_{2n}$ be the number such that the projection of $\tilde{u}$ onto $e_1e_2\cdots e_{2n}$ is $[u]_{2n}e_1e_2\cdots e_{2n}$. If $u=(u_1,\ldots, u_k)\in (\R^{2n})^{\times k}$ and $1\leq j_1,\ldots ,j_l\leq k$ we will also use the notation $[u|j_1,\ldots j_l]_{2n}$ for $[u_0]_{2n}$ where $u_0\in (\R^{2n})^{\otimes k-l}$ is defined as the tensor product of all the $u_j$:s except for $j\in\{j_p\}_{p=1}^{l}$. The complex spin representation of $\C l(\R^{2n})$ is graded and the super trace in this representation vanish on all basis elements but $e_1e_2\cdots e_{2n}$. This fact and a lengthier calculation implies that for any element $v\in \C l(\R^{2n})$ it holds that
\[\tra_{\wedge^{ev}_\C\C^n}(v)-\tra_{\wedge^{odd}_\C\C^n}(v)=(-2i)^n[v]_{2n}.\]

For the natural number $l>0$ we define $\Gamma^l_m\subseteq\{1,2,\ldots, 2m\}^{l}$ as the set of all sequence $\mathbbm{h}=(h_j)_{j=1}^{2l}$ such that $h_j\neq p$ for any $p\leq j$ and $h_j\neq h_p$ for any $j\neq p$. We define $\epsilon_l:\Gamma^j_m\to \{\pm 1\}$ by 
\[\epsilon_l(\mathbbm{h}):=(-1)^{l+\sum_{j=1}^l h_j}.\]

\begin{lem}
\label{tracecalc}
For $x=(x_1,x_2,\ldots x_{2m})\in (\R^{2n})^{\times 2m}$ we have that 
\begin{align*}
&\tra_{\wedge^{ev}_\C\C^n}\left(\prod_{l=1}^{m} c(x_{2l-1})^*c(x_{2l})\right)=(-2)^{n-1}i^n[x_1\otimes x_2\otimes \cdots \otimes x_{2m}]_{2n}+\\
&+(-2)^{n-1}i^n\sum_{l=1}^{m-1}\sum_{\mathbbm{h}\in\Gamma^l_m} \epsilon_l(\mathbbm{h})\left[x|1,h_1,2,h_2,\ldots,l,h_l \right]_{2n}\prod_{p=1}^l\langle x_p,x_{h_{p}}\rangle+\\
&+2^{n-1}\sum_{\mathbbm{h}\in\Gamma^m_m} \epsilon_l(\mathbbm{h})\prod_{p=1}^m\langle x_{p},x_{h_p}\rangle.
\end{align*}
\end{lem}

\begin{proof}
Let us calculate these traces using the relations in the Clifford algebra:
\begin{align*}
\tra_{\wedge^{ev}_\C\C^n}&\left(\prod_{l=1}^{m} c(x_{2l-1})^*c(x_{2l})\right)=\frac{1}{2}\tra_{\wedge^{ev}_\C\C^n}\left(\prod_{l=1}^{m} c(x_{2l-1})^*c(x_{2l})\right)+\\
&+\frac{1}{2}\tra_{\wedge^{ev}_\C\C^n}\left(\left(\prod_{l=1}^{m-1} c(x_{2l})^*c(x_{2l+1})\right)c(x_{2m})^*c(x_1)\right)+\\
&+(-2)^{n-1}i^n[x_1\otimes x_2\otimes \cdots \otimes x_{2m}]_{2n}=\\
=\sum_{j=2}^{2m}(-1)^{j}&\langle x_{1},x_j\rangle \tra_{\wedge_\C^{ev}\C^n}\left(\widehat{c(x_{1})^*c(x_{j})}\right)+(-2)^{n-1}i^n[x_1\otimes x_2\otimes \cdots \otimes x_{2m}]_{2n},
\end{align*}
where $\widehat{c(x_{1})^*c(x_{j})}$ denotes $\prod_{j=1}^{m-1} c(x_{l_{2j-1}})^*c(x_{l_{2j}})$ where $(l_j)_{j=1}^{2m-2}$ is the sequence $1,2,\ldots, 2m$ with the occurences of $1$ and $j$ removed. The sign $(-1)^j$ comes from the numeral of anti-commutations needed to anti-commute the first operator with the $j$:th. Continuing in this fashion one arrives at the conclusion of the Lemma.
\end{proof}

\begin{lem} 
\label{tracalc}
The Chern character of $p_Y$ is given by $\tilde{\nu}^*\ch[p_T]$  and the Chern character of $p_T$ in cyclic homology can be represented by a cyclic $2k$-cycle that in the coordinates on $\R^{2n}\subseteq S^{2n}$ is given by the formula
\begin{align*}
\ch[p_T]&(x_0,x_1,\ldots,x_{2k})=\frac{1}{k!}\tra_{\wedge^*_\C\C^n}\left(\prod_{l=0}^{2k}p_{0}(x_{l})\right)=\\
&=\frac{1}{k!\prod_{l=0}^{2k}(1+|x_l|^2)}\sum_{m=0}^{2k+1}\sum_{0\leq g_1< \cdots < g_m\leq 2k}\tra_{\wedge^{ev}_\C\C^n}\left(\prod_{l=0}^m c(x_{g_l})^*c(x_{g_l+1})\right),
\end{align*}
where we identify $x_{2k+1}=x_0$. 
\end{lem}

\begin{proof}
Define the function $V:\R^{2n}\to \Hom(\wedge^{ev}_\C\C^n,\wedge^{odd}_\C \C^n\oplus \wedge^{ev}_\C \C^n)$ by 
\[V(x)v:=\frac{c(x)v\oplus v}{\sqrt{|x|^2+1}}\in \wedge^{odd}_\C \C^n\oplus \wedge^{ev}_\C \C^n, \quad v\in \wedge^{ev}_\C\C^n.\]
The vector $V$ is defined so that $p_0(x)=V(x)V(x)^*$. Furthermore, observe that $V(x)^*V(y)=c(x)
^*c(y)+1\in \End(\wedge^{ev}_\C\C^n)$. Therefore
\begin{align*}
&\frac{1}{k!}\tra_{\wedge^*_\C\C^n}\left(\prod_{l=0}^{2k}p_{0}(x_{l})\right)=\frac{1}{k!}\tra_{\wedge^{ev}_\C\C^n}\left(V(x_{2k})^*V(x_0)\prod_{l=0}^{2k-1} V(x_j)^*V(x_{j+1})\right)=\\
&=\frac{1}{k!\prod_{l=0}^{2k}(1+|x_l|^2)}\tra_{\wedge^{ev}_\C\C^n}\left(\left(c(x_{2k})^*c(x_0)+1\right)\prod_{l=0}^{2k-1} \left(c(x_j)^*c(x_{j+1})+1\right)\right)=\\
&=\frac{1}{k!\prod_{l=0}^{2k}(1+|x_l|^2)}\sum_{m=0}^{2k+1}\sum_{0\leq g_1\leq \cdots \leq g_m\leq 2k}\tra_{\wedge^{ev}_\C\C^n}\left(\prod_{l=0}^m c(x_{g_l})^*c(x_{g_l+1})\right).
\end{align*}
\end{proof}

\section{Index theory for pseudo-differential operators}

The index of an elliptic pseudo-differential operator can be expressed in terms of local formulas depending on the symbol via the Atiyah-Singer index theorem. In this section we will use elliptic pseudo-differential operators and the Atiyah-Singer index theorem to give an index formula for the degree of a continuous mapping. The theory of pseudo-differential operators can be found in \cite{hormtre}. For an introduction to the Atiyah-Singer index theorem we refer the reader to the survey article \cite{atiyahsinger}. \\

The Atiyah-Singer index theorem, see Theorem $1$ of \cite{atiyahsinger}, states that the index of an elliptic pseudo-differential operator $A$ on the compact manifold $X$ without boundary can be calculated using de Rham cohomology as:
\[\ind (A)=\int_{T^*X} \ch[A]\wedge \pi^*Td(X),\]
where $\ch[A]$ is the Chern character of $A$ and $Td(X)$ is the Todd class of the complexified tangent bundle of $X$ and $\pi:T^*X\to X$ denotes the projection mapping. Since the Chern character is a ring homomorphism we have the following lemma:

\begin{lem}
\label{factoras}
For a smooth projection valued function $p:X\to \Ko$ and an elliptic pseudo-differential operator $A$, the Chern character of the pseudo-differential operator $A_p:=p(A\otimes id)p$ is given by $\ch[A_p]=\ch[A]\wedge\pi^*\ch[p]$.
\end{lem}

Later on, in Theorem \ref{holderindexeven}, the pseudo-differential operator will play a different role compared to the role in the Atiyah-Singer theorem. In the Atiyah-Singer theorem the elliptic pseudo-differential operator defines an element in $K$-theory which pairs with the $K$-homology class whose Connes-Chern character is the Todd class under the isomorphism \eqref{cycde} and gives an index. We will use the elliptic pseudo-differential operator in the dual way as a $K$-homology class that we pair with projections over $C^\infty(X)$ in terms of an index. 

The heuristic explanation of this method is that $K^0(X)$ is a ring and $K^*(T^*X)$ is a $K^*(X)$-module via the projection mapping $\pi:T^*X\to X$. If $A$ is an elliptic pseudo-differential operator, $A$ defines both a $K$-homology class on $X$ and the symbol of $A$ defines an element $[A]\in K^0(T^*X)$. Furthermore, $\ind\!_A(x)=\ind(x\cdot[A])$ for $x\in K^0(X)$, where $\ind:K^0(T^*X)\to \Z$ is the index mapping of Atiyah-Singer. On the level of de Rham cohomology this is exactly the content of Lemma \ref{factoras}.

\begin{sats}
\label{smoothindex}
If $f:X\to Y$ is a smooth mapping between even-dimensional oriented manifolds, $p_Y$ is as in \eqref{localsymboleven} and $A$ is an elliptic pseudo-differential operator on $X$ we have the following degree formula:
\[\ch_0[A]\deg(f)=\ind(A_{p_Y\circ f})-2^{n-1}\ind(A),\]
where $\ch_0[A]$ denotes the constant term in $\pi_*\ch[A]$. If $A$ is of order $0$, the same statement holds for continuous $f$.
\end{sats}

\begin{proof}
By Lemma \ref{factoras}, Theorem \ref{degsatseven} and the Atiyah-Singer index theorem we have that 
\begin{align*}
\ind(A_{p_Y\circ f})=\int_X& \ch[f^*p_Y]\wedge  \pi_*\ch [A]\wedge Td(X)=\\
=&\int_X (2^{n-1}+f^*\rd V_Y)\wedge \pi_*\ch[A]\wedge Td(X)=\\
&\quad=\ch_0[A]\deg(f)+2^{n-1}\int_X  \pi_*\ch[A]\wedge Td(X)=\\
&\quad\quad\quad=\ch_0[A]\deg(f)+2^{n-1}\ind(A).
\end{align*}
Since both $\ind(A_{p_Y\circ f})$ and $\deg(f)$ are well defined homotopy invariants for continuous $f$ when $A$ is of order $0$ the final statement of the Theorem follows.
\end{proof}

To deal with analytic formulas for the mapping degree when $f$ is not smooth will require some more concrete information about Schatten class properties of pseudo-differential operators.

\begin{lem}
\label{firstschatten}
A pseudo-differential operator $b$ of order $-1$ satisfies $b\in \ellL^q(L^2(X))$ for any $q>\dim(X)$.
\end{lem}

Lemma \ref{firstschatten} is proved by using a rather standard technique for pseudo-differential operators. In Theorem \ref{commsats}, when we prove a similar result for H\"older continuous functions we will need some heavier machinery. We include a sketch of the proof of Lemma \ref{firstschatten} just to highlight the difference in methods. Letting $\Delta_X$ denote the second order Laplace-Beltrami operator, the operator $(1-\Delta_X)^{1/2}b$ is of order $0$ whenever $b$ is of order $-1$. Thus $b=(1-\Delta_X)^{-1/2}(1-\Delta_X)^{1/2}b$ and since $(1-\Delta_X)^{1/2}b$ is a bounded operator the Lemma follows if $(1-\Delta_X)^{-1/2}$ is in the Schatten class for any $q>\dim(X)$. This fact follows from the fact that the $k$:th eigenvalue of the Laplacian behaves like $-k^{2/\dim(X)}$ as $k\to \infty$, a statement that goes back to \cite{weyl}. 

\begin{lem}
\label{homgenlemtwo}
The pseudo-differential operator $b$ of order $0$ has an integral kernel $T\in C^\infty(X\times X\setminus D)$ satisfying the estimate $|T(x,y)|\lesssim |x-y|^{-\dim(X)-\epsilon}$ almost everywhere for any $\epsilon>0$, here $D$ denotes the diagonal in $X\times X$.
\end{lem}

Here we use the notation $a\lesssim b$ if there is a constant $C>0$ such that $a\leq Cb$. Observe that the estimate on $T$ only holds for $x \neq y$ so the integral operator defined by $T$ must be realized as a principal value. We will not prove Lemma \ref{homgenlemtwo}, but refer to Theorem $2.53$ of \cite{follandphase}. 

We will use the notation $F^{q,r}_\alpha(X)$ for the space of functions $a\in L^\infty(X)$ such that 
\[J_\alpha(a):(z,w)\mapsto |a(x)-a(y)|/|x-y|^\alpha\] 
is in the mixed $L^P$-space $L^{(r,q)}(X\times X)$. Observe that for any $q$ and $r$ there is an inclusion $C^\alpha(X)\subseteq F^{q,r}_\alpha(X)$ since if $a\in C^\alpha(X)$ then the function $J_\alpha(a)$ is bounded. To be a bit more specific, if we use the notation $|a|_\alpha$ for the H\"older constant of $a$, we have the norm estimate 
\[ \|J_\alpha(a)\|_{L^{(r,q)}}\leq \|J_\alpha(a)\|_{L^{(\infty,\infty)}}=|a|_\alpha,\] 
for any $r,q$. The function space $F^{q,r}_\alpha(X)$ becomes a Banach space in the norm $a\mapsto \|a\|_{L^\infty}+\|J_\alpha(a)\|_{L^{(r,q)}}$. A straight-forward estimate shows that there is a bounded embedding 
\[F^{q,r}_\alpha(X)\subseteq VMO(X)\] 
if $\alpha rq\geq n(r+q)$. The Gagliardo characterization of Sobolev spaces implies that whenever $0<s<1$ and $1<q<\infty$ then 
\[W^{s,q}(X)= F^{q,q}_{\alpha}(X),\] 
where $\alpha=n/q+s$. Since $X$ is compact, this implies that the spaces $F^{q,r}_{\alpha}(X)$ are always equipped with an inclusion in some Sobolev space. We will let $\pi:L^\infty(X)\to \Bo(L^2(X))$ denote the representation given by pointwise multiplication. 

\begin{sats}
\label{commsats}
If $F$ is a pseudo-differential operator of order $0$ on a compact manifold $X$ without boundary and $a\in F^{q,r}_\alpha(X)$ then the operator $[F,\pi(a)]$ is Schatten class of order $q$ and satisfies the norm estimate
\begin{equation}
\label{normestcomm}
\|[F,\pi(a)]\|_{\ellL^q(L^2(X))}\lesssim \|J_\alpha(a)\|_{(r,q)},
\end{equation}
whenever
\begin{equation}
\label{degcond}
\alpha\in]0, 1], \quad q\geq 2\quad\mbox{and}\quad \frac{n}{\alpha}<\frac{rq}{r+q}.
\end{equation}
\end{sats}

Observe that in the limit case $r=\infty$, when $a$ is H\"older continuous, the condition on $q$ is precisely $q>\max(n/\alpha,2)$. The theorem can be reduced to a local claim so we will prove the theorem by describing Schatten class properties of the commutator $[F,\pi(a)]$ in terms of its local nature as in Theorem \ref{russothmeven}.

\begin{proof}
It is sufficient to prove that if $F$ is a properly supported pseudo-differential operator of order $0$ in $\R^n$ and $a\in F^{q,r}_\alpha(\R^n)$ has compact support, then the operator $[F,\pi(a)]$ is Schatten class of order $q$ for $q,r,\alpha$ satisfying the hypothesis of the theorem. Since $F$ is a pseudo-differential operator of order $0$, the operator $F$ can by Lemma \ref{homgenlemtwo} be represented by an integral kernel which is pointwise bounded by $|x-y|^{-n-\epsilon}$ for some $\alpha>\epsilon>0$. Thus the integral kernel of $[F,\pi(a)]$ is bounded by $|a(x)-a(y)||x-y|^{-n-\epsilon}$. 

Let us set $t:=r/q'$ and consider the number 
\[C(\alpha,\epsilon,t,q):=\sup_{y\in B(0,R)}\left(\int_{B(0,R)} |x-y|^{-(n+\epsilon-\alpha)t'q'}\rd V(x)\right)^{1/t'q'}.\]
The function $x\mapsto |x|^{-\beta}$ is integrable whenever $\beta<n$, so $C(\alpha,\epsilon,t,q)$ is finite if $(n+\epsilon-\alpha)t'q'<n$, that is if $n/(\alpha-\epsilon)<rq/(r+q)$. 

While $F$ is properly supported and $a$ has compact support, we can take $\chi,\chi'\in C^\infty_c(\R^n)$ such that $\chi[F,\pi(a)]\chi'=[ F,\pi(a)]$. If $t\geq 1$ and $q\geq 2$ then for some large $R$ 
\begin{align*}
\|[F,\pi(a&)]\|_{\ellL^q(L^2(\R^n))}\lesssim \left(\|\chi J_{-n-\epsilon}(a)\chi'\|_{q',q}\|\chi' J_{-n-\epsilon}(a)\chi\|_{q',q}\right)^{1/2}\lesssim\\
\lesssim& \left(\int _{B(0,R)}\left(\int_{B(0,R)} \frac{|a(x)-a(y)|^{q'}\rd V(x)}{|x-y|^{-(n+\epsilon)q'}}\right)^{q/q'}\rd V(y)\right)^{1/q}=\\
=&\left(\int _{B(0,R)}\left(\int_{B(0,R)} J_\alpha(a)(x,y)^{q'}|x-y|^{-(n+\epsilon-\alpha)q'}\rd V(x)\right)^{q/q'}\rd V(y)\right)^{1/q}\leq\\
\leq&\; C(\alpha,\epsilon,t,q)\left(\int _{B(0,R)}\left(\int_{B(0,R)} J_\alpha(a)(x,y)^{tq'}\rd V(x)\right)^{q/tq'}\rd V(y)\right)^{1/q}= \\
=&\;C(\alpha,\epsilon,t,q)\|J_\alpha(a)\|_{(tq',q)},
\end{align*}
where the fourth inequality follows from H\"older inequality. The condition that $t\geq 1$ is equivalent to that $rq-q-r\geq 0$. However, if $n/\alpha<rq/(r+q)$ then $rq> \frac{n}{\alpha}(q+r)\geq q+r$ since $n\geq \alpha$. Therefore the operator $[F,\pi(a)]$ is Schatten class of order $q$ whenever $n/(\alpha-\epsilon)<rq/(r+q)$. Since $\epsilon$ is arbitrary the Theorem follows.
\end{proof}

Let us digest on the choice of the function spaces $F^{q,r}_\alpha$ in Theorem \ref{commsats}. The reason for choosing these spaces is that they provide a natural setting for using the theorem of Russo. Since our intention is to perform degree calculations in these function spaces, it is of interest to see what happens in the case of Sobolev spaces. If $\alpha=s+n/q$; then the condition \eqref{degcond} of Theorem \ref{commsats} is true if 
\[2n<\alpha q\quad\Longleftrightarrow\quad sq>n.\] 
This means that for any parameters $s,q$ satisfying the sufficient conditions making the commutators Schatten class, the Sobolev embedding implies that $F^{q,q}_\alpha(X)=W^{s,q}(X)\subseteq C^{\alpha_0}(X)$ where $\alpha_0=s-n/q=\alpha-2n/q$. 

\section{Index of a H\"older continuous twist}

In this section we will combine Theorem \ref{connesindexeven} with Theorem \ref{commsats} into index formulas for certain elliptic pseudo-differential operators twisted by H\"older continuous vector bundles. If $p:X\to \Ko(\C^N)$ is a continuous projection-valued function, we will use the notation $E_p$ for the vector bundle over $X$ corresponding to $p$ via Serre-Swan's theorem. 

When $A$ is an elliptic differential operator from the vector bundle $E$ to the vector bundle $E'$, we want to consider the, possibly unbounded, operator $A_p:=p(A\otimes \id_{\C^N})p$ which is called the twist of $A$ by $p$ and acts between the vector bundle $E\otimes E_p$ and $E'\otimes E_p$. However, unless $A$ is of order $0$, we must assume that $p$ is smooth to ensure that $A_p$ is a densely defined Fredholm operator. In the case $p$ is smooth, Lemma \ref{factoras} and the Atiyah-Singer index theorem implies that
\[\ind(A_p)=\int_{T^*X}\pi^*\ch[p]\wedge\ch[A]\wedge \pi^*Td(X),\]
and we also have the identity
\begin{equation}
\label{afind}
\ind(A_p)=\ind(p(A(1+A^*A)^{-1/2}\otimes \id_{\C^N})p),
\end{equation}
because $1+A^*A$ is strictly positive. The right-hand side of \eqref{afind} is well defined for continuous $p$ and, as we will see, it can be calculated for H\"older continuous projections by means of Theorem \ref{connesindexeven} using a certain Fredholm module we associate with $A$. To construct this Fredholm module, we start by defining the odd, self-adjoint operator
\[\tilde{A}:=
\begin{pmatrix}
0& A\\
A^*& 0 \end{pmatrix}\]
on $L^2(X,E'\oplus E)$ which is graded by letting $L^2(X,E')$ be the even part and let $L^2(X,E)$ be the odd part. We define the mapping 
\[\phi:\R\to \R,\quad u\mapsto u(1+u^2)^{-1/2}\quad \mbox{and the operator}\quad F_A:=\phi(\tilde{A}).\] 
The operator $F_A$ is an odd, self-adjoint operator of order $0$. However, the square of the operator $F_A$ can be calculated as
\[F_A^2=\tilde{A}^2(1+\tilde{A}^2)^{-1}=1-(1+\tilde{A}^2)^{-1}\neq 1.\]
To mend the problem $F_A^2\neq 1$, we replace $\phi$ by the function $\tilde{\phi}:\R\to M_2(\C)$ defined as 
\[\tilde{\phi}(u):=\begin{pmatrix}
\phi(u)&(1+u^2)^{-1/2}\\
(1+u^2)^{-1/2}& -\phi(u)\end{pmatrix}.\]
The function $\tilde{\phi}$ satisfies $\tilde{\phi}(u)^2=1$. If we equip $\C^2$ with the grading from the involution
\[\gamma_{\C^2}:=1\oplus (-1),\] 
the operator $\tilde{F}_A:=\tilde{\phi}(\tilde{A})$ is an odd, self adjoint operator on the graded tensorproduct $\C^2\otimes  L^2(X,E\oplus E')$. To simplify notations we set $\Eg:=\C^2\otimes(E'\oplus E)$. The operator $\tilde{F}_A$ does satisfy that $\tilde{F}_A^2=1$. The operator $\tilde{F}_A$ can be written as a matrix of operators as
\begin{equation}
\label{ffour}
\tilde{F}_A:=\begin{pmatrix}
F_A&(1+\tilde{A}^2)^{-1/2}\\
(1+\tilde{A}^2)^{-1/2}& -F_A\end{pmatrix}.
\end{equation}
Since pseudo-differential operators are only pseudo-local, we will use a parameter $t$ to extract the singular part of $\tilde{F}_A$. Let $A(t)$ denote the elliptic differential operator defined from $A$ dilated by the action of $t>0$ on $T^*X$. Set $\tilde{F}(t):=\tilde{F}_{A(t)}$. Define $W_0$ as the smooth pseudo-differential operator given by the orthogonal finite-rank projection onto $\ker (\tilde{A})$ and set $W_0^\perp:=1-W_0$. Observe that, since $\tilde{A}$ is elliptic, we have that $W_0\in C^\infty(X,E'\oplus E)\otimes_{alg} C^\infty(X,(E'\oplus E)^*)$. We define 
\[W:=\begin{pmatrix}
0&W_0\\
W_0&0\end{pmatrix}.\]

\begin{lem}
\label{homlem}
If $A$ is an elliptic differential operator, the operator valued function $t\mapsto \tilde{F}(t)$ satisfies 
\[\|\tilde{F}(t)-\tilde{F}-W\|_{\ellL^p(L^2(X,\mathcal{E}))}=\mathcal{O}(t^{-1})\quad\mbox{as}\quad t\to \infty,\]
where $\tilde{F}$ is the $0$-homogeneous part of $\tilde{F}_A$.
\end{lem}

\begin{proof}
The operator $\tilde{F}(t)-\tilde{F}-W$ is a matrix consisting of terms of the form
\begin{align*}
a_1&=(1+t^2\tilde{A}^2)^{-1/2}-W_0=W_0^\perp(1+t^2\tilde{A}^2)^{-1/2}W_0^\perp=\\
&=t^{-1}W_0^\perp\tilde{A}^{-1}W_0^\perp T_1 \qquad\mbox{and}\\
a_2&=\tilde{A}(t(1+t^2\tilde{A}^2)^{-1/2}-W_0^\perp|\tilde{A}|^{-1}W_0^\perp)=\\
&=\tilde{A}(t(1+t^2\tilde{A}^2)^{-1/2}-|\tilde{A}|^{-1})=t^{-1}W_0^\perp|\tilde{A}|^{-1}W_0^\perp T_2,
\end{align*}
for some $T_1,T_2\in \Bo(L^2(X))$. Since both $a_1$ and $a_2$ are pseudo-differential operators of order lower than $-n/p$, it follows that their $\ellL^p$-norm behaves like $t^{-1}$ as $t\to \infty$ and the Lemma follows.
\end{proof}

As above, $C^\alpha(X)$ denotes the Banach algebra of H\"older continuous functions of exponent $\alpha\in]0,1]$. If $E$ and $E'$ are two smooth vector bundles over $X$ we let $\pi:C^\alpha(X)\to \Bo(L^2(X,E\oplus E'))$ denote the even representation given by pointwise multiplication. Define $\tilde{\pi}:C^\alpha(X)\to \Bo(L^2(X,\Eg))$ by letting $\tilde{\pi}:=\pi\oplus 0$ under the isomorphism $L^2(X,\Eg)\cong \C^2\otimes L^2(X,E'\oplus E)\cong L^2(X,E'\oplus E)\oplus L^2(X,E'\oplus E)$.

\begin{sats}
\label{psumfred}
If $A$ is an elliptic differential operator of positive order between two vector bundles $E$ and $E'$ over $X$ and $q>\max(\dim(X)/\alpha,2)$, then the pair $(\tilde{\pi},\tilde{F}_A)$ is a $q$-summable even Fredholm module over $C^\alpha(X)$.
\end{sats}

\begin{proof}
The operator $\tilde{F}_A$ is an elliptic, self-adjoint pseudo-differential operator of order $0$ and $\tilde{F}_{A}^2=1$. Under the isomorphism $L^2(X,\Eg)= \C^2\otimes  L^2(X,E\oplus E')$ the decomposition \eqref{ffour} implies that
\begin{equation}
\label{ffive}
[\tilde{F}_A,\tilde{\pi}(a)]=
\begin{pmatrix}
[F_A,\pi(a)]&-\pi(a)(1+\tilde{A}^2)^{-1/2}\\
(1+\tilde{A}^2)^{-1/2}\pi(a)&0\end{pmatrix}.
\end{equation}
Since the order of $A$ is positive, these facts together with Theorem \ref{commsats} and Lemma \ref{firstschatten} imply that $(\tilde{\pi},\tilde{F}_A)$ is a $q$-summable even Fredholm module for any $q>\max(\dim(X)/\alpha,2)$.
\end{proof}

Let us represent the pseudo-differential operator $\tilde{F}$ by the integral kernel $\tilde{K}_A$ which is a conormal distribution section of the big $\Hom$-bundle $\Hom(\Eg,\Eg)$. By \eqref{ffour}, we can write $\tilde{K}_A=K_A\oplus (-K_A)$ where $K_A$ is a conormal distribution section of the big $\Hom$-bundle $\Hom(E\oplus E',E\oplus E')$. Observe that $K_A$ is defined by a smooth section $C^\infty(X\times X\setminus D, \Hom(E\oplus E',E\oplus E'))$.

We will use the notation $\Gamma_k$ for the subset of $\{1,2,3,4\}^{2k}$ consisting of all sequences  $(s_l)_{l=1}^{2k}$ satisfying the conditions that $s_1\neq 4$, $s_{2k}\neq 3$ and $s_l=3$ for some $l$ if and only if $s_{l+1}=4$. These conditions are motivated by the form of the commutator \eqref{ffive}. Let $w(I)$ denote the numeral of occurrences of $3$ in $I$. Take $\Gamma_k^w\subseteq \Gamma_k$ as the subset of sequences with $w(I)=w$. To a sequence $I\in \Gamma_k$ we will associate the sequence $\Lambda(I)=(i_l)_{l=1}^{2k}\in\{1,2,\ldots 2k\}^{2k}$ defined by 
\[i_l=\begin{cases}
l,\quad\mbox{if} \quad s_l=1,3\\
l+1,\quad\mbox{if} \quad s_l=2,4\end{cases},\]
where we identify $2k+1$ with $1$. For a sequence $I\in \Gamma_k$ we let 
\[\iota(I):=(-1)^{w(I)+\sum_{l=1}^{2k}i_l-l}.\] 
With a projection-valued function $p:X\to \Ko$ and a sequence $I\in \Gamma_k$ we associate the function 
\begin{equation}
\label{qidef}
Q_{I}^p(x_1,\ldots,x_{2k}):=\tra_{\C^{2^n}}\left(p(x_1)\prod_{l=1}^{2k}p(x_{i_{l}})\right).
\end{equation}
We also define
\[H_A(x_1,\ldots ,x_{2k}):= \suptra_{E\oplus E'}\left(\prod_{l=1}^{2k}K_{A}(x_{i},x_{i+1})\right).\]
A straight-forward calculation implies the following lemma:

\begin{lem}
\label{kalkdia}
For any projection-valued function $p:X\to \Ko(\C^N)$, the following identity holds:
\begin{align*}
\suptra_{(E\oplus E')\otimes \C^N}\left(p(x_1)\prod_{j=1}^{2k}(p(x_{j+1})-p(x_j))K_A(x_{j},x_{j+1})\right)&=\\
=\sum_{(s_l)_{l=1}^{2k}\in \Gamma_k^0} \iota(I)Q_{I}^p(x_1,\ldots,x_{2k}) &H_A(x_1,\ldots ,x_{2k}).
\end{align*}
\end{lem}

This Lemma describes the diagonal terms in the decomposition \eqref{ffive}. To describe the products with the off-diagonal terms, we need some notation for the integral kernels. Set $K_1=K_2=K_A$ and $K_3=K_4=W_0$. For $I\in\Gamma_k$ we define the integral kernel
\[H_{A,I}(x_1,\ldots ,x_{2k}):= \suptra_{E\oplus E'}\left(\prod_{l=1}^{2k}K_{s_l}(x_{i},x_{i+1})\right).\]

\begin{lem}
\label{wsumabsint}
The function
\begin{equation}
\label{wsum}
(x_1,\ldots,x_{2k})\mapsto  \sum_{(s_l)_{l=1}^{2k}\in \Gamma_k^w} \iota(I)Q_{I}^p(x_1,\ldots,x_{2k}) H_{A,I}(x_1,\ldots ,x_{2k})
\end{equation}
is absolutely integrable over $X^{2k}$ for all $w=0,1,\ldots, k$.
\end{lem}

This Lemma is a direct consequence of Theorem \ref{commsats} and Theorem \ref{russotracethmeven} since the sum over $\Gamma^w_k$ corresponds to the sum of the supertraces of the products between $p$ and $2(k-w)$ commutators between $p$ and $K_A$ and $2w$ commutators between $p$ and $W_0$. In fact, when $w>0$ the integral of \eqref{wsum} will be the trace of a finite rank operator. One can decompose the function \eqref{wsum} even further by decomposing $\Gamma_k^w$ into equivalence classes under the equivalence relation $1\sim 2$ and $3\sim 4$, which again will be a decomposition into absolutely integrable functions. 

\begin{sats}
\label{holderindexeven}
Suppose that $A$ is an elliptic differential operator of positive order on the compact manifold $X$ without boundary, acting from $E$ to $E'$. If $p:X\to \Ko(\C^N)$ is a projection valued, H\"older continuous function of exponent $\alpha$ and $2k-1>\max(\dim(X)/\alpha,2)$, the following index formula holds:
\begin{align*}
\ind&(p(F_{A,+}\otimes 1)p)=\langle \cc_k(\tilde{\pi}, \tilde{F}_A),p\rangle_{2k}=\\
&=(-1)^k\int _{X^{2k}} \suptra\left(p(x_1)\prod_{j=1}^{2k}(p(x_{j+1})-p(x_j))K_A(x_{j},x_{j+1})\right)\rd V_{X^{2k}}+\\
&+(-1)^k\sum_{w=1}^k\int _{X^{2k}} \sum_{(s_l)_{l=1}^{2k}\in \Gamma_k^w} \iota(I)Q_{I}^p(x_1,\ldots,x_{2k}) H_{A,I}(x_1,\ldots ,x_{2k})\rd V_{X^{2k}}
\end{align*}
where the trace is taken over $(E\oplus E')\otimes \C^N$ and we identify $x_1=x_{2k+1}$. All of the integrals are absolutely convergent.
\end{sats}

This index formula does in fact work whenever $p\in F^{q,r}_\alpha(X)\otimes \Ko(\C^N)$ for $q,r,\alpha$ satisfying the condition \eqref{degcond} as a consequence of that we are dealing with an explicit index formula that by Theorem \ref{commsats} is well defined. Observe that if the dimension of $X$ is odd, $\ind(p(F_{A,+}\otimes 1)p)=0$ for continuous $p$ since the Atiyah-Singer index theorem implies that it holds for smooth projection-valued functions. 

\begin{proof}
The first equality follows from Theorem \ref{connesindexeven} and Theorem \ref{psumfred} since 
\[\tilde{\pi}(p)(\tilde{F}_{A,+}\otimes 1)\tilde{\pi}(p)=\left(\pi(p)(F_{A,+}\otimes 1)\pi(p)\right)\oplus 0.\]
Recall that, for any $t>0$, the operator $A(t)$ satisfies the same conditions as $A$ so by homotopy invariance of the Chern-Connes character and Lemma \ref{homlem}, 
\begin{align*}
\langle &\cc_k(\tilde{\pi}, \tilde{F}_{A}),p\rangle_{2k}=\langle \cc_k(\tilde{\pi}, \tilde{F}_{A(t)}),p\rangle_{2k}=\\
&=(-1)^k\suptra_{L^2(X,\Eg\otimes \C^N)}\left(\tilde{\pi}(p)[\tilde{F}_{A(t)},\tilde{\pi}(p)]\cdots [\tilde{F}_{A(t)},\tilde{\pi}(p)]\right)=\\
&=(-1)^k\suptra_{L^2(X,\Eg\otimes \C^N)}\left(\tilde{\pi}(p)[\tilde{F}+W,\tilde{\pi}(p)]\cdots [\tilde{F}+W,\tilde{\pi}(p)]\right)+\mathcal{O}(t^{-1}).
\end{align*}
Since $p$ is H\"older continuous, $(p(x_{j+1})-p(x_j))K_A(x_{j},x_{j+1})$ is locally integrable and of finite mixed $L^{(q',q)}$-norm by Theorem \ref{commsats}. Let us set $\tilde{p}:=\tilde{\pi}(p)$ and $\tilde{K}_A^W:= \tilde{K}_A+W$. Theorem \ref{russothmeven} and the calculations above imply the equality
\begin{align*}
&\suptra_{L^2(X,\Eg\otimes \C^N)}\left(\tilde{\pi}(p)[\tilde{F}+W,\tilde{\pi}(p)]\cdots [\tilde{F}+W,\tilde{\pi}(p)]\right)=\\
&\quad=\int _{X^{2k}} \suptra\left(\tilde{p}(x_1)\prod_{j=1}^{2k}\left(\tilde{p}(x_{j+1})\tilde{K}_A^W(x_{j},x_{j+1})-\tilde{K}_A^W(x_{j},x_{j+1})\tilde{p}(x_j)\right)\right)\rd V_{X^{2k}}\\
&\quad=\int _{X^{2k}} \sum_{(s_l)_{l=1}^{2k}\in \Gamma_k} \iota(I)Q_{I}^p(x_1,\ldots,x_{2k}) H_{A,I}(x_1,\ldots ,x_{2k})\rd V_{X^{2k}}.
\end{align*}
While the term $\langle \cc_k(\tilde{\pi}, \tilde{F}_{A}),p\rangle_{k}$ is constant, the second identity stated in the theorem follows from Lemma \ref{kalkdia} by letting $t\to \infty$. 
\end{proof}

\section{Degrees of H\"older continuous mappings}

Returning now to the degree calculations, we will use Theorem \ref{holderindexeven} for a particular choice of differential operator. Assume that $X$ is a compact Riemannian oriented manifold of dimension $2n$ without boundary. The Hodge grading on $\bigwedge^*T^*X\otimes \C$ is defined by the involution $\tau$ defined on a $p$-form $\omega$ by 
\[\tau \omega=i^{p(p-1)+n}*\omega,\] 
where $*$ denotes the Hodge duality. Observe that if $(e_j)_{j=1}^{2n}$ is an oriented, orthonormal basis of the cotangent space, the operator $\tau$ can be written as 
\begin{equation}
\label{tauex}
\tau=\left(\frac{i}{2}\right)^n\prod_{j=1}^{2n}(e_j\wedge-e_j\neg). 
\end{equation}
We let $E_+$ denote the sub-bundle of $\bigwedge^*T^*X\otimes \C$ consisting of even vectors with respect to the Hodge grading and $E_-$ the sub-bundle of odd vectors with respect to the Hodge grading. The operator $\tau$ anti-commutes with $\rd+\rd^*$ so $A=\rd+\rd^*$ is a well defined operator from $E_+$ to $E_-$. The operator $A$ is called the signature operator. Observe that $\tilde{A}=\rd+\rd^*$ as an operator on $\bigwedge^*T^*X\otimes\C$ and $\tilde{A}^2$ is the Laplace-Beltrami operator on $X$. By Theorem \ref{psumfred} the pair $(\tilde{\pi},\tilde{F}_{\rd+\rd^*})$ is a $q$-summable Fredholm module over $C^\alpha(X)$. 

Assume that $f:X\to Y$ is a H\"older continuous function where $Y$ is a $2n$-dimensional oriented manifold. We can choose an open subset $U\subseteq Y$ such that there is a diffeomorphism $\nu:U\to B_{2n}$. In fact, by the closed mapping lemma, since $X$ is compact and $Y$ Hausdorff, the mapping $f$ is open. Therefore, it is possible to choose $U$ such that there is an open set $U_0\subseteq X$ satisfying $U\subseteq f(U_0)$ and $E_+$ and $E_-$ are trivial over $U_0$.

Taking $\tilde{\nu}:Y\to \C^{n}$ as the Lipschitz continuous function defined in \eqref{nuexteven} we can for $k>n/\alpha$ define the integrable function $\tilde{f}_k\in L^1(X^{2k})$ as:
\begin{align}
\label{tildef}
&\tilde{f}_k(x_1,\ldots, x_{2k}):=\\
\nonumber
&= 2^{1-n}\sum_{I\in \Gamma_k} \iota(I)Q_{I}^{p_T}(\tilde{\nu}f(x_1),\ldots, \tilde{\nu}f(x_{2k})) H_{\rd+\rd^*,I}(x_1,\ldots ,x_{2k}),
\end{align}
where the second quantity is calculated through Lemma \ref{tracalc}. The kernel $H_{\rd+\rd^*,I}$ is in general quite hard to find. In local coordinates, the operator $K_{\rd+\rd^*}$ will be similar to a Riesz transform. The operator $W_0$ is the projection onto the finite-dimensional space of harmonic forms on $X$. We will demonstrate this by calculating the kernels $H_{\rd+\rd^*,I}$ explicitly on $S^{2n}$ in the next section.

\begin{sats}
\label{holderdegreeeven}
Suppose that $X$ and $Y$ are smooth compact connected oriented manifolds without boundary of dimension $2n$ and $f:X\to Y$ is H\"older continuous of exponent $\alpha$. When $k>n/\alpha$ the following integral formula holds:
\begin{align*}
\deg(f)&=2^{-n}\langle  \cc_k(\tilde{\pi}, \tilde{F}_{D}),f^*([p_Y]-2^{n-1}[1])\rangle_{2k}=\\
&=\frac{1}{2}\left((-1)^k\int _{X^{2k}}\tilde{f}_k(x_1,\ldots, x_{2k})\rd V_{X^{2k}}-\sign(X)\right)
\end{align*}
where $\tilde{f}_k$ is as in \eqref{tildef}.
\end{sats}

\begin{proof}
While $F_{\rd+\rd^*}=(\rd+\rd ^*)(1+\Delta)^{-1/2}$, we have that 
\[\ind (F_{\rd+\rd^*,+})=\ind (\rd +\rd^*)=\sign(X).\] 
The operator $\rd +\rd^*$ satisfies $\ch_0[\rd +\rd^*]=2^nL_0(T^*X)=2^n$, since the constant term in the $L$-genus is $1$. Therefore Theorem \ref{smoothindex} and Theorem \ref{holderindexeven} implies that 
\begin{align*}
\deg(f)&=2^{-n}\ind ((p_Y\circ f\otimes \id)F_{\rd+\rd^*,+}(p_Y\circ f\otimes \id))-2^{-n}\ind (F_{\rd+\rd^*,+})=\\
&=2^{-n}\langle \cc_k(\tilde{\pi}, \tilde{F}_A),f^*p_Y\rangle_{2k}-\frac{1}{2}\sign(X)=\\
&=\frac{1}{2}(-1)^k\int _{X^{2k}}\tilde{f}_k(x_1,\ldots, x_{2k})\rd V_{X^{2k}}-\frac{1}{2}\sign(X).
\end{align*}
\end{proof}

A couple of remarks on the choice of $A$ as the signature operator are in order. This choice is rather superfluous since any pseudo-differential operator $A$ of order $1$ with $\ch_0[A]\neq 0$ will give a degree formula similar to that in Theorem \ref{holderdegreeeven}. If one can find an invertible $A$ on $X$ such that $\ch_0[A]\neq 0$, the formula of Theorem \ref{holderdegreeeven} would be much simpler since $\tilde{f}_k$ will not contain any contributions from $\Gamma_k\setminus \Gamma_k^0$. 

The signature operator has been studied on Lipschitz manifolds, see \cite{teleman}, which gives an analytic degree formula for H\"older continuous mappings between even-dimensional Lipschitz manifolds. On Lipschitz manifolds, the Atiyah-Singer theorem is replaced by Teleman's index theorem from \cite{teltva}. Of course, there are some analytic difficulties in the proof of Theorem \ref{commsats} for Lipschitz manifolds, that more or less manifests themselves on a notational level.   

If $X$ and $Y$ are compact manifolds, we define the space $F^{q,r}_\alpha(X,Y)$ in terms of an embedding $Y\subseteq \R^N$ and give it the topology induced from $F^{q,r}_\alpha(X,\R^N)$. As a corollary of the previous theorem we obtain the following result:

\begin{cor}
If $X$ and $Y$ are compact, oriented manifolds of the same dimension and $f\in F^{q,r}_\alpha(X,Y)$ then 
\[|\deg(f)|\leq \tilde{c}^{2k}_{X,Y}+c_{X,Y}^{2k}\|J_\alpha(f)\|_{(r,q)}^{2k},\]
whenever $2k\geq q$ and $\alpha,q,r$ satisfies condition \eqref{degcond}.
\end{cor}

\section{Example on $S^{2n}$}

Let us end this paper by writing down the integral kernels in the case of a H\"older continuous function $f:S^{2n}\to Y$ where $Y$ is a compact, connected $2n$-dimensional manifold without boundary. We will compare the method of using pseudo-differential operators from Theorem \ref{holderdegreeeven} with that of using Henkin-Ramirez kernels from \cite{mgoffodd}. The degree formula of \cite{mgoffodd} is based on the usage of the cocycle \eqref{oddcycle}. 

\subsection{Degrees on $S^{2n}$ using Theorem \ref{holderdegreeeven}}
To apply Theorem \ref{holderdegreeeven} the kernels $H_{\rd+\rd^*,I}$ need to be calculated. Define $U_0:L^2(S^{2n}, \bigwedge T^*S^{2n}\otimes \C)\to L^2(\R^{2n},\bigwedge T^*\R^{2n}\otimes \C)$ by pulling back along the mapping $\lambda:\R^{2n}\to S^{2n}$ defined by 
\[\R^{2n}\ni x\mapsto \left(\frac{|x|^2-1}{|x|^2+1},\frac{2x}{|x|^2+1}\right)\in S^{2n}\subseteq \R^{2n+1}\]
and equipping $\R^{2n}$ with the pull-back metric. Since the metric is positive definite we can define the unitary mapping $U:L^2(S^{2n}, \bigwedge T^*S^{2n}\otimes \C)\to L^2(\R^{2n})\otimes\C^{2^{2n}}$ by composing $U_0$ with the unitary mapping $L^2(\R^{2n}, \bigwedge T^*\R^{2n}\otimes \C)\to L^2(\R^{2n})\otimes\C^{2^{2n}}$ defined by the metric. Clearly, we have that $UK_{\rd+\rd^*}U^*$ is a pseudo-differential operator on $\R^{2n}$ with symbol $\xi\mapsto i(\xi\wedge-\xi\neg)/\sqrt{2}|\xi|$, if we identify $\C^{2^{2n}}$ with $\bigwedge \C^{2n}$. For any $a\in C^\alpha(S^{2n})$ we have that $UaU^*=a\circ \lambda$.

Let us find the integral kernel $K_1$ of $UK_{\rd+\rd^*}U^*$. Since the symbol of $K_1$ is a homogeneous function that commutes with the $SU(2n)$-action on $\bigwedge\C^{2n}$, the integral kernel $K_1$ is given by 
\[K_1(x,y)=c_n\frac{(x-y)\wedge-(x-y)\neg}{\sqrt{2}|x-y|^{2n+1}}\]
and the constant $c_n=(n-1)!/\pi^n$ is calculated in Chapter III in \cite{stein}. So the operator defined by $K_1$ is a matrix of Riesz transforms. While the harmonic forms on $S^{2n}$ are spanned by the constant function and the volume form, the kernel $W_0$ is the constant projection $\bigwedge T^*S^{2n}\to \C\oplus\bigwedge^{2n}T^*S^{2n}$. So we can write the kernel $K_3$ of $UW_0U^*$ as
\[K_3(x,y)=g(x)g(y)(1+\prod_{j=1}^{2n}e_j\wedge e_j\neg),\]
where 
\[g(x):=\frac{2n\pi^nn!(n-2)!}{(2n+1)!(1+|x|^2)^{n}}.\] 
Using these expressions for $K_1$ and $K_3$ it follows from \eqref{tauex} that 
\begin{align*}
H_{\rd+\rd^*,I}(x_1,\ldots ,x_{2k})&= \suptra\left(\prod_{l=1}^{2k}K_{s_l}(x_{l},x_{l+1})\right)=\tra\left(\tau\prod_{l=1}^{2k}K_{s_l}(x_{l},x_{l+1})\right)= \\
&=\sum_{\sigma\in S_{2(k-w)}}\sum_{m=0}^n  i^n\sign(\sigma) H_{\sigma,m,I}(x_1,\ldots ,x_{2k}),
\end{align*}
when $w(I)=w$ and where
\begin{align*}
H_{\sigma,m,I}(x_1,\ldots &,x_{2k}):=c_n^{2k-m}\prod_{s_l\neq 1,2} g(x_l)g(x_{l+1})\prod_{l=1, s_l=1,2}^m \frac{\langle e_j, x_{\sigma(l)}-x_{\sigma(l)+1}\rangle}{|x_{\sigma(l)}-x_{\sigma(l)+1}|^{2n+1}}\cdot\\
&\quad\cdot\prod_{l=m+1, s_l=1,2}^{k} \frac{\langle x_{\sigma(l)}-x_{\sigma(l)+1}, x_{\sigma(l+1)}-x_{\sigma(l+1)+1}\rangle}{|x_{\sigma(l)}-x_{\sigma(l)+1}|^{2n+1}|x_{\sigma(l+1)}-x_{\sigma(l+1)+1}|^{2n+1}}.
\end{align*}
By this notation we mean that an element $\sigma$ in the symmetric group $S_{2(k-w)}$, on $2(k-w)$ elements, acts on the indices $l$ such that $s_l=1,2$. Here we use the notation $\langle\cdot,\cdot\rangle$ for the scalar product. For a H\"older continuous function $f:S^{2n}\to Y$ we take an open set $U\subseteq Y$ such that there is a diffeomorphism $\nu:U\to B_{2n}$ and consider the H\"older continuous function $f_0:=\tilde{\nu} f\lambda:\R^{2n}\to S^{2n}$, where $\tilde{\nu}$ is as in equation \eqref{nuexteven}. The signature of a sphere is $0$ so Theorem \ref{holderdegreeeven} and \eqref{tildef} implies the following degree formula.

\begin{prop}
If $f:S^{2n}\to Y$ is H\"older continuous of exponent $\alpha$, where $Y$ is a closed oriented connected $2n$-dimensional manifold, the degree of $f$ can computed from the formula  
\begin{align*}
&\deg(f)=\\
&=\frac{(-1)^k k!}{2^n}\sum_{w=0}^k\int _{\R^{4nk}}\sum_{I\in \Gamma_k^w} \sum_{\sigma\in S_{2(k-w)}}\sum_{m=0}^n i^n \sign(\sigma)c_n^{2k-m}\iota(I)\cdot\prod_{s_l\neq 1,2} g(x_l)g(x_{l+1})\cdot\\
&\qquad\qquad\cdot \ch[p_T](f_0(x_1),f_0(x_{i_1}),\ldots f_0(x_{i_{2k}}))\cdot\prod_{l=1, s_l=1,2}^m \frac{\langle e_j, x_{\sigma(l)}-x_{\sigma(l)+1}\rangle}{|x_{\sigma(l)}-x_{\sigma(l)+1}|^{2n+1}}\cdot\\
&\qquad\qquad\cdot \prod_{l=m+1, s_l=1,2}^{k} \frac{\langle x_{\sigma(l)}-x_{\sigma(l)+1}, x_{\sigma(l+1)}-x_{\sigma(l+1)+1}\rangle}{|x_{\sigma(l)}-x_{\sigma(l)+1}|^{2n+1}|x_{\sigma(l+1)}-x_{\sigma(l+1)+1}|^{2n+1}}\;\rd V_{\R^{2nk}}.
\end{align*}
\end{prop}

\subsection{Degrees on $S^{2n}$ using Theorem $4.4$ of \cite{mgoffodd}}

If we attempt using the degree formula of \cite{mgoffodd} to a function $f:S^{2n}\to Y$ we must in some way change the dimension. We will do so by finding a strictly pseudo-convex domain $\Omega$ in $\C^{n+1}$ such that $\partial\Omega=S^{2n}\times S^1$ and use the degree formula of \cite{mgoffodd} to calculate the degree of $f\times \id:S^{2n}\times S^1\to Y\times S^1$. To find such domain $\Omega$ we define the function $\rho\in C^\infty(\C^{n+1})$ as 
\[\rho(z_1,\ldots, z_{n+1}):=4|1-z_1z_2|^2+|z|^2-3.\]
We let $\Omega:=\{z\in \C^{n+1}:\rho(z)<0\}$. Clearly, $\rho$ is strictly pluri-subharmonic in $\Omega$ and $\rd \rho\neq 0$ on $\partial\Omega$. Furthermore 
\[\rho(z)= 4(1-2\Re(z_1z_2)+|z_1z_2|^2)+|z|^2-3.\]
So $\Omega$ is a relatively compact strictly pseudo-convex domain in $\C^{n+1}$ with smooth boundary. Writing down the function $\rho$ in its real argument verifies that $\Omega\cong B_{2n+1}\times S^1$ and $\partial\Omega\cong S^{2n}\times S^1$.

Let us find the Henkin-Ramirez kernel $H_R$ of the boundary $\partial\Omega$. In order to do this we will use the Fornaess embedding theorem, see \cite{forn}. This approach to construct the Henkin-Ramirez kernel goes as follows; if we find a proper, holomorphic mapping from $\Omega$ into a convex domain, we obtain the holomorphic support function of $\partial\Omega$ by pulling back the support function in the convex domain. Define the holomorphic function $\Psi:\C^{n+1}\to \C^{n+2}$ by $(z_1,z_2, \ldots, z_{n+1})\mapsto (2(1-z_1z_2),z_1,z_2, \ldots, z_{n+1})$. Clearly $\Omega$ is mapped in to the ball of radius $3$, $3B_{2n+4}$,  and $\Omega=\Psi^{-1}(3B_{2n+4})$. Therefore $\Omega$ has the support function 
\begin{equation}
\label{nevher}
a(z,\zeta):=(\Psi(\zeta)-\Psi(z))\cdot \overline{\Psi(\zeta)}=\bar{\zeta}\cdot(\zeta-z)+4(\zeta_1\zeta_2-z_1z_2)\overline{\zeta_1\zeta_2}.
\end{equation}
Thus the integral kernel of the Henkin-Ramirez kernel for $S^{2n}\times S^1$ is given by 
\[H_R(z,\zeta)\rd V_{S^{2n}\times S^1}=\frac{1}{(2\pi i)^n}\frac{s\wedge(\bar{\partial} _\zeta s)^{2(n-1)}}{a(z,\zeta)^{n}},\]
where
\[s(z,\zeta)=\partial |\zeta|^2+2(1-\overline{\zeta_1\zeta_2})(z_2\rd \zeta_1+\zeta_1\rd \zeta_2).\]
Because of \eqref{nevher} the function $a(z,\zeta)$ is symmetric in the sense that for $z,\zeta\in \partial\Omega$ we have $a(z,\zeta)=\overline{a(\zeta,z)}$. However, the kernel $H_R$ is not symmetric as is seen from the expression for $s$.

Concluding the case of a H\"older continuous function $f:S^{2n}\to Y$, we take an open subset $U\subseteq Y$ diffeomorphic to a ball and choose a diffemorphism $\nu:U\times (S^1\setminus\{pt\})\to B_{2n+1}$. If we let $f_1:S^{2n}\times S^1\to Y\times S^1$ be the H\"older continuous function constructed by extending $\nu(f\times \id_{S^1})$ as in \eqref{nuexteven}, Theorem $4.4$ of \cite{mgoffodd} implies the next Proposition.

\begin{prop}
If $f:S^{2n}\to Y$ is H\"older continuous of exponent $\alpha$, where $Y$ is a closed oriented connected $2n$-dimensional manifold, the degree of $f$ can computed from the formula  
\begin{align*}
\deg(f)=\frac{(-1)^{n}}{(2\pi i)^{2kn}}\int_{(S^{2n}\times S^1)^{2k+1}} &\sum_{\sigma\in S_{2(2k-1)}}\sum _{l=0}^{2k-1} c_{l,\sigma}\prod_{i=1}^l\langle f_1(z_{\sigma(2i-1)}),f_1(z_{\sigma(2i)})\rangle \cdot\\
&\cdot\prod_{j=0}^{2k} \frac{s(z_{j},z_{j+1})\wedge(\bar{\partial} _{z_{j+1}} s(z_{j},z_{j+1}))^{2(n-1)}}{a(z_{j},z_{j+1})^{n}}.
\end{align*}
\end{prop}

Here $c_{l,\sigma}=(-1)^l2^{n-l-1}\epsilon_l(\sigma)$ and $\epsilon_l(\sigma)$ is the order parity of $\sigma$. For details see \cite{mgoffodd}.

\end{document}